\documentclass[reqno,10pt]{amsart}
\usepackage[english]{babel}
\usepackage{amsthm}
\usepackage{amsmath}
\usepackage{amssymb}
\usepackage[latin1]{inputenc}
\usepackage[normalem]{ulem}% overlines
\usepackage{tikz}
\usepackage{tikz-3dplot}
\usepackage{pgfplots}
\usepackage{enumerate}

\usepackage[all]{xy}
\usepackage{dsfont}
\usepackage[bookmarksnumbered,colorlinks]{hyperref}
\usepackage{graphics}
\usepackage{graphicx}
\usepackage{enumerate}

\def\bb#1\eb{\textcolor{blue}
{#1}} %
\def\br#1\er{\textcolor{red}
{#1}} %
\newcommand{\df}{{\rm d}}

\newcommand{\R}{\mathds R}
\newcommand{\N}{\mathds N}

\newcommand{\C}{\mathcal{C}}
\newcommand{\LL}{\mathds L}
\newcommand{\HH}{\mathds H}

\newcommand{\be}{\begin{equation}}
\newcommand{\ee}{\end{equation}}

   \def\br#1\er{\textcolor{red}{#1}} %
      \def\bb#1\eb{\textcolor{blue}{#1}} %

\title[Wind Riemannian structures]{Wind Riemannian spaceforms \\ and Randers metrics of \\ constant flag curvature}

\author[M. A. Javaloyes]{Miguel Angel Javaloyes}
\address{Departamento de Matem\'aticas, \hfill\break\indent
Universidad de Murcia, \hfill\break\indent
Campus de Espinardo,\hfill\break\indent
30100 Espinardo, Murcia, Spain}
\email{majava@um.es}

\author[M. S\'anchez]{Miguel S\'anchez}
\address{Departamento de Geometr\'{\i}a y Topolog\'{\i}a, Facultad de Ciencias, \hfill\break\indent
 Universidad de Granada,\hfill\break\indent
 Campus Fuentenueva s/n,
 \hfill\break\indent 18071 Granada, Spain}
\email{sanchezm@ugr.es}

\date{30.11.2016}


\begin{document}
\newtheorem{thm}{Theorem}[section]
\newtheorem{prop}[thm]{Proposition}
\newtheorem{lemma}[thm]{Lemma}
\newtheorem{cor}[thm]{Corollary}
\theoremstyle{definition}
\newtheorem{defi}[thm]{Definition}
\newtheorem{notation}[thm]{Notation}
\newtheorem{exe}[thm]{Example}
\newtheorem{conj}[thm]{Conjecture}
\newtheorem{prob}[thm]{Problem}
\newtheorem{rem}[thm]{Remark}
\newtheorem{conv}[thm]{Convention}
\newtheorem{crit}[thm]{Criterion}

\begin{abstract}
Recently, {\em wind Riemannian structures (WRS)} have been introduced  as a  generalization of Randers and Kropina metrics. They are constructed from the natural data for Zermelo navigation problem, namely, a Riemannian metric   $g_R$ and a vector field  $W$ (the {\em wind}), where, now, the restriction of mild wind $g_R(W,W)<1$ is dropped. 

Here,  the models of  WRS spaceforms of constant flag curvature are determined. Indeed, 
the celebrated classification of Randers metrics of constant flag curvature  by Bao, Robles and Shen \cite{BRS}, extended to the Kropina case in the  works by Yoshikawa, Okubo and Sabau \cite{YoOk07, YoSa}, can be used to obtain the  local classification. For the global one, a suitable result on completeness for WRS yields  
the complete simply connected models. 
In particular, any of the local models in the Randers classification does admit an extension to a unique model of wind Riemannian structure, even if it cannot be extended 
as a complete Finslerian manifold.

Thus,  WRS's  emerge as the natural framework for the analysis of Randers spaceforms and, prospectively,  {\em wind  Finslerian structures} would become important for other global problems too. 
For the sake of completeness, a brief  overview about WRS (including  a  useful link with the conformal geometry of a class of relativistic  spacetimes) is also provided.

\end{abstract}

\maketitle

%\newpage
%\tableofcontents
%\newpage

\section{Introduction}

Wind Riemmanian structures (WRS's) are generalizations of the indicatrices of both,  Randers and Kropina metrics on a manifold $M$, introduced in \cite{CJS} for several purposes. The simplest one is that 
%such a WRS $\Sigma$ 
they provide a natural framework for modelling 
Zermelo navigation problem. Namely, consider the motion of a zeppelin: its engine is able to develop a maximum speed with respect to the air, modelled by the indicatrix of a Riemanian metric $g_R$, but the air is moving with respect to earth with a (time-independent) velocity modelled by a vector field $W$. So, the maximum speed of the zeppelin with respect to earth is modelled by $\Sigma=S_R+W$. This is the indicatrix of a Randers metric when $g_R(W,W)<1$,   but when $g_R(W,W)=1$, it is the indicatrix of a Kropina metric, a singular Finsler metric in the sense that it is not defined in all the directions;  the geodesics of such metrics solve Zermelo problem of finding the fastest path between two prescribed points. In the more general case in that no restriction on $g_R(W,W)$ is imposed, $\Sigma$ is a WRS and,  as proven in \cite{CJS}, such a WRS admits a notion of geodesic which also solves Zermelo navigation. However, there is a deeper motivation for  studying  WRS's  because of  the existence of a link between the geometry of such WRS's and the conformal geometry of a class of relativistic spacetimes, those which are standard with a space-transverse Killing
vector field, or {\em SSTK spacetimes}.

 The existence of a fruitful correspondence between the geometry of Randers manifolds and the conformal geometry of stationary  spacetimes (a particular case of SSTK ones) has already been  pointed out and systematically exploited in \cite{CaJaMa, CJS, FHS, GHWW,JLP15} and others.
However, the importance of the correspondence in the case of general WRS's becomes especially useful for the Finslerian framework. Indeed, the existence of apparently singular Finslerian elements for WRS's  (such as regions with a Kropina metric) is reinterpreted from the Lorentzian viewpoint in a completely non-singular way. In particular, geodesics for WRS's can be seen as projections of lightlike pregeodesics  of SSTK spacetimes. Moreover, geodesic completeness of WRS's becomes equivalent to the global hyperbolicity  of the spacetime.    Notice that both, lightlike pregeodesics and global hyperbolicity are conformally invariant elements in Lorentzian Geomety.

 Our main aim along the present note will be to show an application of WRS's to 
 %the understanding of 
 Randers metrics of constant flag curvature (CFC).  The complete classification of these manifolds was a landmark in Finslerian Geometry, obtained by  Bao, Robles and Shen \cite{BRS}. The solutions admit a neat description, which can be summarized by saying that $g_R$ must have constant curvature and $W$ must be  a  homothetic vector field.
 From a global viewpoint, however, there is  a  striking difference with the classical (simply connected) models of Riemannian manifolds of constant curvature:  in the Randers case, some of the  models  are necessarily incomplete.  Namely,  some local models cannot be extended to a complete one,  and only its  inextensibility can be claimed. Nevertheless, we will see that this comes from the fact that $g_R(W,W)$ may not remain bounded by 1, and we will show that all the local models admit a unique extension as a complete simply connected WRS. Moreover, taking into account the classification of Kropina metrics of CFC by   Yoshikawa and Okubo \cite{YoOk07} and Yoshikawa and Sabau \cite{YoSa} we will provide both, the local and global classification of WRS's.
 
 This article is organized as follows. In Section \ref{s2},  the motivations and necessary results on WRS's are briefly summarized.  
In Section \ref{s3}, the classification of WRS's  of CFC is achieved. For this purpose, we summarize first the known results in the Randers \cite{BRS} and Kropina cases \cite{YoOk07, YoSa} (subsection \ref{s3.1}), and  show how these results can be used for the local classification  (subsection \ref{s3.2}). The global classification is obtained in subsection \ref{s3.3}. Here, the key issue is completeness. First, we  prove a result of independent interest on completeness for WRS's,  Theorem \ref{th_compl}; the subtleties of this result are stressed in Example \ref{ex1}.  The global classification is reached in Theorem \ref{thm:globalresult}, by using the previous  result on completeness, the local result and a specific study of the case when $W$ is properly homothetic. 
Finally, in  the appendix (Section \ref{s4}), we summarize some known results on Killing and homothetic vector  fields.  The results in subsection \ref{s4.1}  are useful to understand when Randers models can be naturally extended in  WRS's, and those in subsection~\ref{s4.2} for a neat description of the Kropina case and the region of  transition between mild and strong wind.

\section{A Finslerian overview on Wind Riemannian Structures}\label{s2}

\subsection{Generalizing Randers and Kropina metrics}
Let $(g_R,W)$ be Zermelo data  in a manifold $M$, namely, $g_R$ is a Riemannian metric and $W$, a vector field on $M$. Assume that  $|W|_R<1$, where $|\cdot |_R$ denotes the pointwise $g_R$-norm, and consider the Randers metric $F$ whose indicatrix $\Sigma$ at $p\in M$   is 
obtained by displacing the $g_R$-indicatrix $S_R$  with the vector  $W_p$, that is, $\Sigma = S_R+ W_p$. This metric is the  key of Zermelo navigation problem (the $F$-geodesics solve it), and it can be written as:
$$
F(v_p)= \frac{1}{\Lambda(p)} \left(\sqrt{\Lambda(p) |v_p|_R^2+ g_R(v_p,W_p)^2}-g_R(W_p,v_p)\right) , \quad \hbox{where} \; \Lambda= 1-|W|_R^2,
$$ 
 for any $v_p\in T_pM$. 
Clearly, this expression crashes when $|W_p|_R=1$, as $\Lambda$ vanishes. However, we can  rewrite it by removing the root  from  the numerator and, then, 
$\Lambda$  from  the denominator:
\begin{equation}\label{conicF}
F(v_p)= \frac{|v_p|_R^2}
{g_R(W_p,v_p)+\sqrt{\Lambda(p) |v_p|_R^2+ g_R(v_p,W_p)^2}} = \frac{|v_p|_R^2}
{g_R(W_p,v_p)+\sqrt{h(v_p,v_p)}},
\end{equation}
where 
\begin{equation}\label{hmetric}h(v_p,v_p)= \Lambda(p) |v_p|_R^2+ g_R(v_p,W_p)^2.\end{equation}
This expression for $F(v_p)$ makes sense for an arbitrary wind $W$ even if $\Lambda$ vanishes, suggesting a possibility for the description of the case $|W|_R\geq  1$. However, a  caution should be taken into account: both, $F(v_p)$ and the expression for $h$ (which lies inside a root) should be nonnegative. So, restrict the domain of $F$ by imposing:
%will be restricted to  a subset of $TM$:
\begin{equation}\label{Arestriction}
\hbox{if} \; \Lambda(p)\leq 0 \; \hbox{restrict to} \; 
\left\{\begin{array}{ll}h(v_p,v_p)\geq 0, \; \hbox{and} \\ g_R(W_p,v_p)>0.
\end{array}\right.
\end{equation}
These restrictions have the following meaning (see Figure \ref{figure1}). 
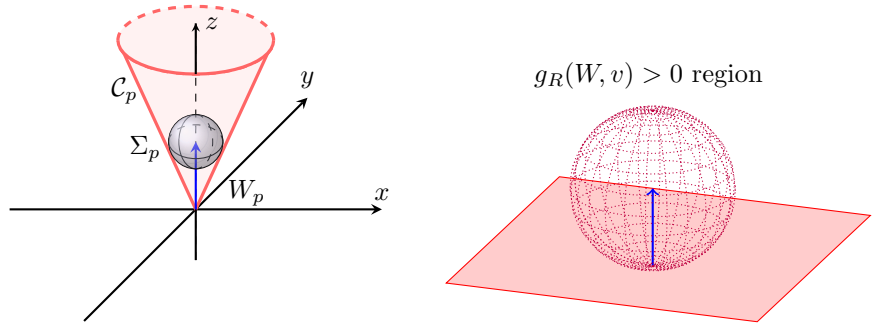
\begin{figure}
\begin{tikzpicture}[scale=0.9,x=0.5cm,y=0.5cm,z=0.3cm,>=stealth]
% The axes
\draw[->,thick] (xyz cs:x=-5.5) -- (xyz cs:x=5.5) node[above] {$x$};
%\draw[->] (xyz cs:y=-1.5) -- (xyz cs:y=5.5) node[right] {$z$};
\draw[->,thick] (xyz cs:z=-5.5) -- (xyz cs:z=5.5) node[above] {$y$};
%\draw[line width=2pt]   (-2,5) -- (0,0) -- (2,5);
\filldraw[color=red!60, fill=red!5, very thick] (-2.3,5) -- (0,0) -- (2.3,5);
\filldraw[color=red!60, fill=red!5, very thick,dashed] (0,5) circle (1.15cm and 0.5cm);
\filldraw [color=red!60, fill=red!5, very thick,domain=180:360] plot ({2.3*cos(\x)}, {1*sin(\x)+5});
\node at (0,7)(s){$h(v,v)\geq 0$ and $g_R(W,v)>0$ region};
\draw[dashed,->] (xyz cs:y=-1.5) -- (xyz cs:y=5.5) node[right] {$z$};
\draw[thick] (0,-1.5) -- (0,0);
 \draw[thick,blue,->] (0,0) -- (0,2);
 \draw[thick] (0,4) -- (0,5.5);
 \node at (1.5,0.5)(s){$W_p$};
 \node at (-1.5,1.8)(s){$\Sigma_p$};
 \node at (-2.1,3.5)(s){$\C_p$};
  \draw (-0.8,2) arc (180:360:0.4cm and 0.25cm);
    \draw[dashed] (-0.8,2) arc (180:0:0.4cm and 0.25cm);
    \draw (0,2.8) arc (90:270:0.25cm and 0.4cm);
    \draw[dashed] (0,2.8) arc (90:-90:0.25cm and 0.4cm);
    \draw (0,2) circle (0.4cm);
    \shade[ball color=blue!10!white,opacity=0.20] (0,2) circle (0.4cm);
\end{tikzpicture}
\tdplotsetmaincoords{70}{110}
\begin{tikzpicture}
	[scale=1.1,
		tdplot_main_coords,
		curve/.style={purple,densely dotted}]

	\node at (4,1.4,3)(s){$g_R(W,v)>0$ region};
	\filldraw[
        draw=red,%
        fill=red!20,%
    ]          (0,-1.2,0)
            -- (4,-1.2,0)
            -- (4,2.8,0)
            -- (0,2.8,0)
            -- cycle;
             \draw[thick,blue,->] (0,0,-1) -- (0,0,0);
            
            \coordinate (O) at (0,0,0);
	
	\foreach \angle in {-90,-75,...,90}
	{
		%calculate the sine and cosine of the angle
		\tdplotsinandcos{\sintheta}{\costheta}{\angle}%

		%define a point along the z-axis through which to draw
		%a circle in the xy-plane
		\coordinate (P) at (0,0,\sintheta);

		%draw the circle in the main frame
		\tdplotdrawarc[curve]{(P)}{\costheta}{0}{360}{}{}
		
		%define the rotated coordinate frame based on the angle
		\tdplotsetthetaplanecoords{\angle}
		
		%draw the circle in the rotated frame
		\tdplotdrawarc[curve,tdplot_rotated_coords]{(O)}{1}{0}{360}{}{}
	}

\end{tikzpicture}
\caption{\label{figure1}  The  diagram on the left represents  the domain of $F_p$ when $\Lambda(p)<0$. In this case, we consider $g_R$ the Euclidean metric in $\R^3$ and $W=(0,0,2)$. The condition $h(v,v)\geq 0$ determines that the region must lie in the two cones of $h$ and the condition $g_R(W,v)>0$ that it must be contained in the half-space $z>0$; so, the cone $\C_p$ is selected.  The diagram on the right represents the indicatrix and the domain of $F_p$, $z>0$, in the Kropina case when $\Lambda(p)=0$ and $W=(0,0,1)$. }
\end{figure}
Formula 
\eqref{hmetric} shows that $h$ is a signature-changing metric, which becomes Riemannian in the region of mild wind $|W|_R<1$, degenerate when the wind is critical $|W|_R=1$, and Lorentzian of signature $(+,-,\dots , -)$ when the wind is strong $|W|_R>1$. 
While in the region $|W|_R<1$ one 
recovers a Randers metric with indicatrix $F^{-1}(1)=\Sigma$ ($=S_R+W)$, in the region $|W|_R=1$ (i.e., $\Lambda=0$) one has a Kropina metric $
\alpha^2/\beta$ with $\alpha=|\cdot|_R$, 
$\beta= 2 g_R(W,\cdot )$. Indeed the restriction $g_R(W_p,v_p)>0$ selects   the (pointwise) tangent open half-space $A_p$ where $F$ becomes
 positive, and one can still regard 
 $\Sigma=S_R+W$ as the indicatrix of $F$ (up to the ``singular'' vector $0_p$), Figure \ref{figureindic}.
 \begin{figure}
\tdplotsetmaincoords{70}{110}
\begin{tikzpicture}[scale=1.8,tdplot_main_coords]
    %\draw[thick] (0,0,0) -- (4,0,0); %node[anchor=north east]{$x$};
    %\def\x{.5}
    %5\draw[thin] (0,0,0) -- ({1.2*\x},{sqrt(3)*1.2*\x},0) node[below] {$y=\sqrt{3}x$};
    \filldraw[
        draw=red,%
        fill=red!20,%
    ]          (0,0,0)
            -- (4,0,0)
            -- (4,4,0)
            -- (0,4,0)
            -- cycle;
       \filldraw[
        draw=blue,%
        fill=blue!20,%
    ]          (1,2,0)
            -- (2,2,0)
            -- (2,3,0)
            -- (1,3,0)
            -- cycle;
          \filldraw[
        draw=blue,%
        fill=blue!20,%
    ]          (0.5,0.5,0)
            -- (1.5,0.5,0)
            -- (1.5,1.5,0)
            -- (0.5,1.5,0)
            -- cycle;
            \filldraw[
        draw=blue,%
        fill=blue!20,%
    ]          (2.5,0.25,0)
            -- (2.5,1.25,0)
            -- (3.5,1.25,0)
            -- (3.5,0.25,0)
            -- cycle;
           \draw[thick] (1,1,0) circle (0.3cm and 0.15cm);
           \filldraw[blue] (1,1.1,0) circle (0.5pt);
           \draw[thick] (3,0.7,0) circle (0.3cm and 0.15cm);
           \filldraw[blue] (3,0.395,0) circle (0.5pt);
           \draw[thick] (1.5,2.5,0) circle (0.3cm and 0.15cm);
           \filldraw[blue] (1.5,2.1,0) circle (0.5pt);
           \draw[thick,red] (1.5,2.1,0) -- (1.5+0.5,2.1+0.4,0);
           \draw[thick,red] (1.5,2.1,0) -- (1.5-0.5,2.1+0.2,0);
    %\draw[thick] (0,0,0) -- (0,4,0);% node[anchor=north west]{$y$};
    %\draw[very thick,->,blue] (0,0,0) -- (0,0,2) node[left,anchor=south]{$t$};
\end{tikzpicture}
\caption{\label{figureindic}The three possibilities for indicatrices of the WRS in a $2$-dimensional example.}
\end{figure}
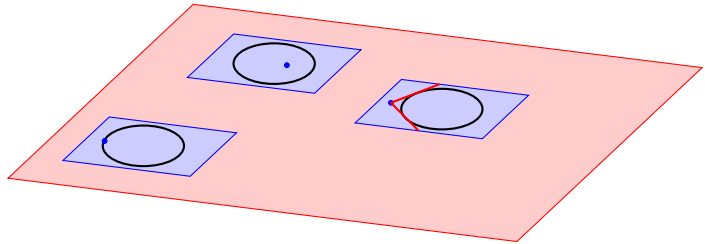
 In the region 
$|W|_R>1$, the Lorentzian metric $h$ determines  at every tangent space  two (lightlike) cones, and the restrictions \eqref{Arestriction} have
a neat meaning: the allowed $v_p$ must belong  either to one of these two $h$-cones $\C_p$ (the cone selected by $g_R(W_p,v_p)>0$) or to the interior 
region $A_p$ determined by it  (see Figure \ref{figure1}).  So, $F$ can be regarded as a ``conic'' Finsler 
metric defined only on the tangent vectors satisfying 
\eqref{Arestriction}. Notice that, now, 
$\Sigma=S_R+W$ includes the ``indicatrix" $F^{-1}(1)$. Indeed, $
\Sigma_p \cap \C_p$ divides $\Sigma_p$  into  two pieces, and $F^{-1}(1)$ corresponds with one of them; namely, the (strongly) convex one, when looking from infinity into the cone region. Easily, one can check that the other piece is equal to $F_l^{-1}(1)$ where 
\begin{equation}\label{lorentzF}
F_l(v_p):=   \frac{|v_p|_R^2}
{g_R(W_p,v_p)-\sqrt{h(v_p,v_p)}}.
\end{equation}
Recall that $F_l^{-1}(1)$
 is  (strongly) concave and, so,  $F_l$ will be 
called a {\em Lorentz-Finsler} metric.
Notice that $F_l$ is univocally determined by\footnote{This property does not hold for the general wind Finslerian structures studied in \cite{CJSwind}.} $F$; consequently, most of our computations will deal only with $F$. However, $F_l$  has a nice interpretation for Zermelo navigation under strong wind: the indicatrix of $F_l$ provides the {\em minimum} speed of the moving object at each allowed direction.

Let us summarize the previous approach and introduce suitable conventions:
\begin{enumerate}
\item Given a (connected) manifold $M$, a {\em wind Riemannian structure (WRS)} is any hypersurface $\Sigma \subset TM$ which can be expressed as $\Sigma= S_R+W$ for some vector field $W$ and Riemannian metric $g_R$ (both  univocally determined)\footnote{In \cite{CJSwind},  general wind Finslerian structures are introduced in a more abstract way, and wind Riemannian ones are then regarded as a particular case of that definition. However, both approaches are clearly equivalent  (see \cite[Prop. 2.13]{CJSwind}). }. At each $p\in M$, $\Sigma_p$ encloses an open domain $B_p$ which will be called the {\em unit ball at $p$.}

\item Such a  $\Sigma$ determines the (possibly signature-changing) metric $h$ in \eqref{hmetric}, as well as a  domain $A:=\cup_{p\in M}A_p$ included in the slit tangent bundle $TM\setminus\{\mathbf{0}\}$, defined by choosing each $A_p$ as follows: $A_p=T_pM\setminus\{0\}$ 
in the region of mild wind ($\Lambda(p)>0$), $A_p$ is the open half space determined by \eqref{Arestriction} in the region of critical wind ($\Lambda(p)=0$), and $A_p$ is the interior of the solid cone determined by \eqref{Arestriction} in the region $M_l$ of strong wind ($\Lambda(p)<0$). 

\item $\Sigma$ also determines:
\begin{enumerate}[(i)]
\item A conic Finsler metric  $F:
A\rightarrow \R$.  
\item A Lorentz-Finsler metric  $F_l:
A_l\rightarrow \R$,  where $A_l$ is the union of the pointwise domains $A_p$ in the region $M_l$ of strong wind.
\end{enumerate}
Moreover, the following three extensions of the domains $A, A_l$ will be used when necessary: 
\begin{enumerate}[(a)]
\item $F$ and $F_l$ are extended continuously on 
$M_l$ to the $h$-cones $\C_p \subset T_pM\setminus\{0\}$,

\item  $F_l$ is extended as $\infty$
outside the region 
of strong wind, that is, in $TM\setminus(\{\mathbf{0} \}\cup TM_l)$, and 

\item in the region of critical wind ($\Lambda (p)=0)$), we define $F(0_p)=F_l(0_p)=1$  (even though,  necessarily, such a choice is discontinuous, see below). 
\end{enumerate}
\end{enumerate} 
%\br We will say that $A_E=\{\text{the closure of $A$ in $TM\setminus 0$}\}\cup \{0_p: \Lambda (p)=0\}$ is the extended domain. Observe that we have extended $F$ and $F_l$ to this region. Indeed, we will refer to these extensions as the extended conic pseudo-Finsler metrics denoted resp. as $\bar{F}:A_E\rightarrow \R$ and $\bar{F}_l: A_E\rightarrow \R$. Observe that the domain of definition of these conic Finsler metrics is not open. \er

\subsection{Wind curves and the appearance of Lorentzian geometry} In the framework of Zermelo navigation, consider two points $p,q\in M$ and a WRS $\Sigma$ which provides the maximum velocities at each direction. If a moving object going from $p$ to $q$ is represented by the curve $\gamma: [t_0,t_1]\rightarrow M$, where $t_0, t_1$ are,  respectively,  the instant of departure from $p$ and arrival to $q$, then its velocity  at each instant $t$ must be an allowed one, that is:
\begin{equation}\label{windcurve}
F(\dot\gamma(t)) \leq 1 \leq 
F_l(\dot\gamma(t)), \qquad \qquad \forall t\in[t_0,t_1].
\end{equation}
Any curve in $M$ satisfiying these inequalities will be called a {\em wind curve}. Recall that \eqref{windcurve} assumes implicitly that $\dot\gamma(t)$  lies in the domains of $F$ and $F_l$ explained at the end of the last subsection, so:
(a) $\dot\gamma(t)$ is allowed to belong not only to the unit ball $B_{\gamma(t)}$ 
but also to its closure (in the slit tangent bundle) $
\Sigma_{\gamma(t)}$, (b) as $F_l\equiv \infty$ outside $M_l$, the last inequality 
in \eqref{windcurve} imposes no restriction when the wind is not strong, and (c) the zero velocity $0_p$ is excluded  in both, the region of strong wind (because it is not an allowed velocity) and  
the region of mild wind (by convenience, analogous to the restriction to regular curves in Riemannian Geometry); however, the zero-velocity is allowed in the region of critical wind, as it has a special meaning there, namely, it is  the minimum allowed velocity in the direction of the wind. %(indeed, this will be unavoidable %when geodesics are considered). 
The arrival time of the moving object $t_1-t_0$ is bounded by the $F$ and $F_l$ lengths of the curves, that is:
$$
\ell_F(\gamma):= \int_{t_0}^{t_1}F(\dot\gamma(t))dt \leq t_1-t_0 \leq \ell_{F_l}(\gamma):= \int_{t_0}^{t_1}F_l(\dot\gamma(t))dt .
$$ 
Obviously, when $\gamma$ is reparametrized so that $F(\dot\gamma(t))\equiv 1$, this corresponds to a moving object which  uses the maximum possible velocity along the trajectory of $\gamma$ so that it  spends the  minimum possible travel time along that trajectory; analogously, in the case that  $\gamma$ lies entirely in the region of strong wind, a reparametrization with $F_l(\dot\gamma(t))\equiv 1$ corresponds to minimum velocity and  maximum travel time.

A better insight is obtained by considering the graph $\{(t,\gamma(t)), t\in [t_1,t_2]\}\subset \R\times M$ of the wind curve. % (see Figure \ref{graphfigure}). 
%\begin{figure}
%\begin{center}
%\begin{tikzpicture}
%%\begin{scope}[shift={(0,-3)},out=-20,in=160,relative]
%%    \filldraw[color=blue!30, fill=blue!10, very thick] (0,0) to (4,0.5) to (5,3) to (1,2.5) to cycle;
%%  \end{scope}
%  \filldraw[color=blue!30, fill=blue!10, very thick] (0,-3) -- (4,-2.5) -- (5,0) -- (1,-1) -- cycle;
%%  \begin{scope}[shift={(0,-3)},out=-20,in=160,relative]
%%    \filldraw[color=blue!30, fill=blue!10, very thick] (0,0) to (4,0.5) to (5,3) to (1,2.5) to cycle;
%%  \end{scope}
%  \filldraw[blue] (1,-2) circle (1.5pt);
%  \filldraw[blue] (4,-1) circle (1.5pt);
%  \filldraw[blue] (4,1) circle (1.5pt);
%  \draw[thick] (4,-1) -- (4,2);
%  \draw[thick,red,dashed] (1,-2)  .. controls (2,-1) and (3,-3)  .. (4,-1);
%  \draw[thick,red] (1,-2)  .. controls (2,0) and (3,-3)  .. (4,1);
%  \node[left] at (1,-2)(s){$\gamma(t_1)$};
%  \node[right] at (4,-1)(s){$\gamma(t_2)$};
%   \node[right] at (4,1)(s){$(t_2,\gamma(t_2))$};
%  \node at (2.8,0)(s){$(t,\gamma(t))$};
%   \node at (3,-2.35)(s){$\gamma(t)$};
%    \node[right] at (4,2)(s){$\R$};
%    \node at (0.5,-0.2)(s){$M$};
% \end{tikzpicture} 
% \end{center}
% \caption{\label{graphfigure} The graph of the curve $\gamma$}
%\end{figure}
The allowed velocities for $\gamma$ at each instant $t\in\R$ and each point $p$  are represented by 
\begin{equation}\label{windgraph}
\{(1,v_p)\in T_{(t,p)}(\R\times M): F(v_p) \leq 1 \leq 
F_l(v_p)\}.
\end{equation}
Now, notice that the half lines in $T_{(t,p)}(\R\times M)$ which start at $0$ and cross any of these allowed velocities provide a solid cone on  
$T_{(t,p)}(\R\times M)$  (recall Figure \ref{figurecones}).  Thus, the WRS yields naturally a cone structure on all $\R\times M$. This cone structure is invariant in the $t$-coordinate, as neither the wind  $W$  nor the  metric $g_R$ are time-dependent. A simple computation shows that the cone structure is equal to the (future) cone structure associated with a Lorentzian metric $g$ on $\R \times M$, namely:
\begin{equation}
\label{eg}
g=- (\Lambda \circ \pi) \df t^2+\pi^*\omega\otimes \df t+\df
t\otimes \pi^*\omega+\pi^*g_0
\end{equation}
where $\pi:\R\times M\rightarrow M$ is the natural projection and:
\begin{equation}
\label{egbis}
\Lambda=1-|W|^2_R, \qquad \omega=-g_R(W,\cdot ), \qquad  g_0=g_R.
\end{equation}
Some comments on this Lorentzian  metric are in order:
\begin{itemize}
\item The metric $g$ is Lorentzian with signature $(-,+,\dots, +)$. At each 
 $(t,p)$ the non-zero tangent vectors $(\tau, v_p)\in T_{(t,p)}(\R\times M)$ 
satisfying $g((\tau, v_p),(\tau, v_p))=0$ (resp. $<0$,  $>0$) are called {\em 
lightlike} (resp. {\em timelike}, {\em spacelike}). 
The lightlike vectors at each $(t,p)$ are distributed in two cones. One of them, the {\em future-directed lightlike cone}, 
contains tangent vectors 
with $\tau>0$; the other one is called the {\em past-directed lightlike cone}.  The future-directed lightlike cone plus the corresponding {\em future-directed timelike vectors} (those inside the solid cone) contain the set \eqref{windgraph} determined by the wind curves. 
\item The natural vector field $K=\partial_t$ is  Killing for $g$. When $\Lambda>0$ (resp. $=0,<0$), $K$ is future-directed timelike (resp. future-directed lightlike, spacelike), see Figure \ref{figurecones}. The natural projection %\sout{$t: M\times \R \rightarrow \R$} \br 
$t: \R\times M \rightarrow \R$ %\er 
is a {\em time function} because it is strictly increasing on any future-directed timelike or lightlike curve; its slices $t=$constant are {\em spacelike hypersurfaces} (they inherit a Riemannin metric). According to \cite{CJSwind}, these spaces are called {\em standard with  a spacelike-transverse Killing vector field},
 or just
{ SSTK} spacetimes. They include important families of relativistic spacetimes, as the {\em standard stationary ones} (those with $\Lambda>0$), which correspond to the Randers case under our approach.
\item The SSTK spacetime $(\R\times M,g)$ will be 
canonically associated with a  WRS. However, the only 
relevant properties of the spacetime for our purposes 
will be the {\em conformally 
invariant} ones\footnote{So, sometimes a different 
representative of the 
conformal class of the SSTK 
may be preferred. For 
example,  the 
normalization $\Lambda\equiv 
1$ was chosen in the case of Randers 
metrics and standard 
stationary spacetimes studied 
in \cite{CJS}. (This is the reason why we preferred to write the metric $g_0$ in \eqref{eg} even if  it is taken $g_0=g_R$ later.)}. 
Indeed, any conformal metric  $g^*=\Omega\cdot g$, $\Omega>0$, will have the same lightlike cones as $g$. Moreover, it is  well-known  that if two Lorentzian metrics $g,g^*$ share the same timelike cones then they are (pointwise) conformally related through some function $\Omega>0$ \cite{BEE, MinSan}.     
\end{itemize}
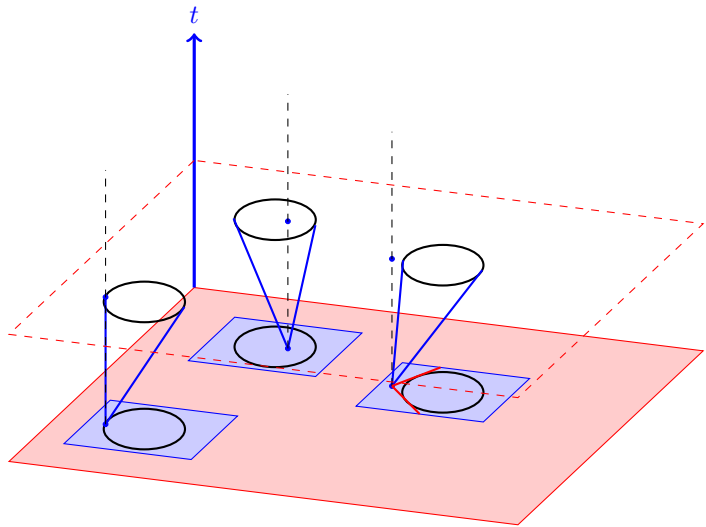
\begin{figure}
\tdplotsetmaincoords{70}{110}
\begin{tikzpicture}[scale=1.8,tdplot_main_coords]
    %\draw[thick] (0,0,0) -- (4,0,0); %node[anchor=north east]{$x$};
    %\def\x{.5}
    %5\draw[thin] (0,0,0) -- ({1.2*\x},{sqrt(3)*1.2*\x},0) node[below] {$y=\sqrt{3}x$};
    \filldraw[
        draw=red,%
        fill=red!20,%
    ]          (0,0,0)
            -- (4,0,0)
            -- (4,4,0)
            -- (0,4,0)
            -- cycle;
       \filldraw[
        draw=blue,%
        fill=blue!20,%
    ]          (1,2,0)
            -- (2,2,0)
            -- (2,3,0)
            -- (1,3,0)
            -- cycle;
          \filldraw[
        draw=blue,%
        fill=blue!20,%
    ]          (0.5,0.5,0)
            -- (1.5,0.5,0)
            -- (1.5,1.5,0)
            -- (0.5,1.5,0)
            -- cycle;
            \filldraw[
        draw=blue,%
        fill=blue!20,%
    ]          (2.5,0.25,0)
            -- (2.5,1.25,0)
            -- (3.5,1.25,0)
            -- (3.5,0.25,0)
            -- cycle;
            %first cone
           \draw[thick] (1,1,0) circle (0.3cm and 0.15cm);
           \filldraw[blue] (1,1.1,0) circle (0.5pt);
           \draw[thick] (1,1,1) circle (0.3cm and 0.15cm);
           \filldraw[blue] (1,1.1,1) circle (0.5pt);
           \draw[thick,blue] (1,1.1,0)  -- (1+0.01,1.1+0.22,1);
           \draw[thick,blue] (1,1.1,0)  -- (1-0.05,1.1-0.445,1-0.05);
           %second cone
           \draw[thick] (3,0.7,0) circle (0.3cm and 0.15cm);
           \filldraw[blue] (3,0.395,0) circle (0.5pt);
           \draw[thick] (3,0.7,1) circle (0.3cm and 0.15cm);
           \filldraw[blue] (3,0.395,1) circle (0.5pt);
           \draw[thick,blue] (3,0.395,0)  -- (3,0.395+0.62,1);
           \draw[thick,blue] (3,0.395,0)  -- (3,0.395,1);
           %third cone
           \draw[thick] (1.5,2.5,0) circle (0.3cm and 0.15cm);
           \filldraw[blue] (1.5,2.1,0) circle (0.5pt);
           \draw[thick] (1.5,2.5,1) circle (0.3cm and 0.15cm);
           \filldraw[blue] (1.5,2.1,1) circle (0.5pt);
           \draw[thick,blue] (1.5,2.1,0) -- ((1.5,2.1+0.71,1);
           \draw[thick,blue] (1.5,2.1,0)  -- (1.5,2.1+0.088,1);
           % orbitas de K
            \draw[dashed] (1,1.1,0) -- (1,1.1,2);
             \draw[dashed] (3,0.395,0)  -- (3,0.395,2) ;
              \draw[dashed] (1.5,2.1,0) -- (1.5,2.1,2);
              \draw[thick,red] (1.5,2.1,0) -- (1.5+0.5,2.1+0.4,0);
           \draw[thick,red] (1.5,2.1,0) -- (1.5-0.5,2.1+0.2,0);
   % \draw[thick] (0,0,0) -- (0,4,0); node[anchor=north west]{$y$};
    \draw[very thick,->,blue] (0,0,0) -- (0,0,2) node[left,anchor=south]{$t$};
     \draw[
        draw=red,dashed%
    ]          (0,0,1)
            -- (4,0,1)
            -- (4,4,1)
            -- (0,4,1)
            -- cycle;
\end{tikzpicture}
\caption{\label{figurecones}The three possibilities for the cones  of the associated SSTK spacetime. At the tangent space of each $p$, the interesection of the $(n+1)$-dimensional cone with the hyperplane $dt_p=1$ projects into the indicatrix of a Randers, Kropina or wind-Riemannian metric.}
\end{figure}
The correspondence between the WRS and the SSTK spacetime is very fruitful for both, our Finslerian problem  and the geometry of relativistic spacetimes. Indeed, this happens even in the particular case $\Lambda>0$, where  the WRS is just a Randers metric and the SSTK spacetime, a standard stationary spacetime (see the detailed study in \cite{CJS} and further developments such as \cite{FHS,JLP15}). 
However, the correspondence becomes crucial for general WRS.  Indeed, the study of WRS through more classical Finslerian elements such as the metrics $F$, $F_l$ presents the important drawback of having ``singular'' elements (as the 0 vector 
in the Kropina region) or discontinuous elements (as the sudden jump in the structure of $A_p$ when $p$ varies from the region of mild wind to the non-mild one). This reason underlies the difficulties to develop general Zermelo navigation in spite of attempts such as \cite{Caratheodory}. 
Nevertheless, the spacetime viewpoint allows both: (i) a description in terms of completely regular and non-singular elements (the  metric $g$) and (ii) the possibility to derive results for WRS by using   results or techniques in the well established conformal theory of spacetimes. Indeed,  the general problem of Zermelo navigation can be described and solved satisfactorily by using this correspondence  \cite{CJSwind}. We will not go deeper here in this correspondence and refer to the exhaustive study in \cite{CJSwind}; however, we would like to point out that some of the results  below were obtained by using it.

\subsection{Balls, geodesics and completeness} \label{sec:balls}
Let $\Sigma$ be a WRS and,  for any $p,q\in M$,  let $C^{\Sigma}_{p, q}$ denote the set of all the wind curves from $p$ to $q$.
The {\em forward} and {\em backward} {\em wind balls}  of center $p_0\in M$ and radius $r>0$ 
 associated with the WRS $\Sigma$ are, resp:
\begin{align*}
&B^+_{\Sigma}(p_0,r)=\{x\in M:\ \exists\   \gamma\in C^{\Sigma}_{p_0, x}, \text{ s.t. }   r=b_\gamma-a_\gamma \, \text{and} \;  \ell_F(\gamma)<r<\ell_{F_l}(\gamma)\},\\
&B^-_{\Sigma}(p_0,r)=\{x\in M:\ \exists\   \gamma\in C^{\Sigma}_{x, p_0}, \text{ s.t. }  r=b_\gamma-a_\gamma \, \text{and} \;  \ell_F(\gamma)<r<\ell_{F_l}(\gamma)\}.\\
\intertext{These balls are open \cite[Remark 5.2]{CJSwind} and their closures are called {\em (forward, backward) closed wind balls}, denoted  $\bar{B}^\pm_{\Sigma}(p_0,r)$. Between these two types of  balls,  the {\em forward} and {\em backward  c-balls} are defined, resp., by:} 
&\hat{B}^+_{\Sigma}(p_0,r)=\{x\in M:\ \exists\  \gamma\in C^{\Sigma}_{p_0, x},\text{ s.t. } r=b_\gamma-a_\gamma \, \text{(so,} \;  \ell_F(\gamma)\leq r\leq\ell_{F_l}(\gamma) )  \}
%\cup\{p_0:\text{ if $0_{p_0}\in \Sigma_{p_0}$}\}
,\\
&\hat{B}^-_{\Sigma}(p_0,r)=\{x\in M:\ \exists\ \gamma\in C^{\Sigma}_{x, p_0},\text{ s.t. } r=b_\gamma-a_\gamma \, \text{(so,} \; \ell_F(\gamma)\leq r\leq\ell_{F_l}(\gamma))  \} 
%\cup\{p_0:\text{ if $0_{p_0}\in \Sigma_{p_0}$}\}, 
 \end{align*}
for  $r> 0$; for $r=0$, by convention $\hat{B}^\pm_{\Sigma}(p_0,0)=p_0$ (so that, 
consistently with our conventions, if $0_{p_0}\in \Sigma_{p_0}$ then $p_0 \in \hat B^\pm_\Sigma(p_0,r)$ for all $r\geq 0$). If $\Sigma$ comes from a Randers metric,  then  $B^\pm_{\Sigma}(p_0,r)$
coincides with the usual (forward or backward) open balls. However, even in the Riemannian case, one may have 
$B^\pm_{\Sigma}(p_0,r)\subsetneq \hat B^\pm_{\Sigma}(p_0,r) \subsetneq \bar B^\pm_{\Sigma}(p_0,r)$ (put $M=\R^2\setminus \{(1,0)\}$, $p_0=(0,0)$, $r=2$). 

Starting at these notions of balls, geodesics can be defined as follows.  A   wind  curve $\gamma: I=[a,b]\to M$,  $a<b$,   is called a
{\em unit extremizing geodesic} if
\begin{equation}\label{eunitgeodesic}
 \gamma(b)\in \hat{B}_\Sigma^+(\gamma(a),b-a)\setminus B_\Sigma^+(\gamma(a),b-a) .  \end{equation}
%for every $t\in (a,b]$. 
Then, a curve is an {\em extremizing geodesic}  if it is an affine    %increasing 
reparametrization  of a unit extremizing geodesic, and it is a {\em geodesic} if it is locally an {\em extremizing geodesic}.

The geodesics of a WRS coincide, up to a reparametrization, with the projection on $M$ of the future-directed lightlike geodesics of the associated SSTK spacetime $(\R\times M,g)$. This allows one to prove that a  curve $\gamma: I\rightarrow M$ is a geodesic if an only if it lies in one of the following cases: 
\begin{enumerate}
\item $\gamma$ is a geodesic of the conic Finsler metric $F$. In this case, $\dot\gamma(t)$ lies always in $A$ (it cannot belong to its boundary) and $\gamma$ may lie in the regions of mild, critical or strong wind, eventually crossing them; moreover,  $\gamma$ minimizes locally the $F$-length in a natural sense.

\item $\gamma$ is a geodesic of the Lorentz-Finsler metric $F_l$. In this case, $\dot\gamma(t)$ lies always in $A_l$ (it cannot belong to its boundary) and $\gamma$ is entirely contained in the region of strong wind $M_l$; moreover,  $\gamma$ maximizes locally the $F_l$-length in a natural sense.

\item  $\gamma$ is an {\em exceptional geodesic}, that is, $\gamma$ is constantly equal to some point $p_0$ 
with $\Lambda (p_0)=0$ and $d\Lambda(v_{p_0})=0$ for all 
$v_{p_0}$ %\in T_{p_0}M$ 
satisfying $g_R(W_{p_0},v_{p_0})=0$.

\item $\gamma$ is included in the closure of $M_l$ and it satisfies: (i) whenever $\gamma$ remains in $M_l$,  it is 
a lightlike  pregeodesic of the Lorentzian metric $h$  in \eqref{hmetric},  reparametrized so that $F(\dot\gamma)\equiv  F_l(\dot\gamma)$ is a constant $c>0$ (that is, $\gamma$ is a {\em boundary geodesic} in $M_l$, with velocity in the boundary of each solid cone $\bar A_p, p\in M_l$), and (ii) $\gamma$ can reach the boundary $\partial M_l$ (which is included in the critical region $\Lambda=0$) only at  isolated points $s_j\in I, j=1,2,...$, where\footnote{Recall that the equality $F(\dot\gamma)=  F_l(\dot\gamma)=c>0$ close to $s_j$ is compatible with $\dot\gamma(s_j)=0$ because of the Kropina character of $\Sigma$ at $\gamma(s_j)$.}  
$\dot\gamma(s_j)=0$,  $d\Lambda$ does not vanish on all the $g_R$-orthogonal to $W_{\gamma(s_j)}$  and the second derivative of $\gamma$ (in one and then in any coordinates) is continuous and does not vanish at $s_j$.  

\end{enumerate}

The WRS is {\em  (geodesically) complete} when  its inextendible geodesics are defined on all $\R$; it is called {\em forward} (resp. {\em backward}) {\em complete} when only the upper (resp. lower) unboundedness of their intervals of definition  is required.  
  Given $\Sigma$, its  {\em reverse WRS} is $-\Sigma= S_R-W$. Clearly, the reverse parametrization of  a geodesic for $\Sigma$ becomes a geodesic for $-\Sigma$ and $\Sigma$ is forward complete iff $-\Sigma$ is backward complete. From now on,  completeness will mean forward and backward completeness. We have the following characterization extracted from \cite[Prop. 6.4]{CJSwind}. 
\begin{thm}\label{c63}
Let $(M,\Sigma)$ be a WRS. The following properties are equivalent:
\begin{enumerate}[(i)]
\item $\Sigma$ is geodesically complete, \item $B^+_\Sigma(x,r)$
and $B^-_\Sigma(x,r)$ are precompact for every $x\in M$ and $r>0$.
\item $\hat B^+_\Sigma(x,r)$ and $\hat B^-_\Sigma(x,r)$ are
compact for every $x\in M$ and $r>0$.
\end{enumerate}
In particular, if $M$ is compact then $\Sigma$ is  geodesically  complete.
\end{thm}
The following result of completeness will be used throughout the text.
\begin{thm}\label{JVresult}
Let $(M,\Sigma)$ be a complete WRS and $W$ a complete homothetic field of the associated conic pseudo-Finsler metrics $F$ and $F_l$. Then,  $\Sigma+W$ is a complete WRS.
\end{thm}
\begin{proof}
It is a straightforward consequence of \cite[Theorem 1.2]{JV}.
\end{proof}

%(3) Anyadir el resultado de completitud Javal Vitorio que usamos un par de veces luego (de acuerdo)

The equality between c-balls and closed balls in a (connected) Finslerian 
manifold becomes equivalent to its  convexity (in the sense that any $p,q
\in M$ can be joined by a geodesic of minimum length). In general,  a WRS is 
called {\em w-convex} when this equality between balls hold. In a w-convex WRS,  an $F$-extremizing wind geodesic 
from $p$ to $q$ will  exist whenever there exists a
wind curve starting from $p$ and ending at $q$. %(for two prescribed points $p,q\in M$).   
A complete WRS is always w-convex, and even 
in this case some of their points might be non-connectable through wind curves
(see Figure \ref{Randers-KropinaFig}).
%\begin{exe}\label{exa1}  Let us consider the Randers Kropina metric on $M=\R^2$ ($g_R$ usual scalar product) 
%with $W_{(x,y)}=f(x)\partial_x$, $f(x) \equiv 0$ for $|x|\geq 4$, $f$ with increasing from $0$ to $1$ on $[-4,-3]$,
% decreasing from $1$ to $-1$ on $[-3,-1]$,  increasing from $-1$ to $1$ on $[-1,1]$,  decreasing from $1$ to $-1$ on $[1,3]$ and increasing from $-1$ to $0$ on $[3,4]$. It is geodesically complete and, thus, w-complete even though  wind curves starting at the strip  $x\in [-3,3]$ (resp. $(-3,3)$, $(-\infty, -1]$, $[1,\infty)$) cannot leave this region and they can be incomplete for the Euclidean metric. Exceptional geodesics appear in $x=\pm 3, \pm 1$.
% \end{exe}
 \begin{figure}
 \centering
\begin{tikzpicture}[scale=0.5]
\draw[thick,->] (-7.5,0) -- (7.5,0) node[anchor=north west] {x axis};
\draw[thick,->] (0,-3.5) -- (0,3.5) node[anchor=south east] {y axis};
%\draw (-3,0) .. controls (-2,1) and (-1,-1) .. (1,1);
%\draw [cyan, xshift=4cm] plot [smooth, tension=2] coordinates { (-3,0) (-2,1) (-1,-1) (1,1)};
\foreach \x in {-6,-3,0,3,6}
    \draw (\x cm,1pt) -- (\x cm,-1pt) node[anchor=north] {$\x$};
\foreach \y in  {-3,3}
    \draw (1pt,\y cm) -- (-1pt,\y cm) node[anchor=east] {$\y$};
    \draw[dashed] (-6,3) -- (-6,-3);
    \draw[thick,blue] (-6,0) circle (1cm);
    \draw[dashed] (-3,3) -- (-3,-3);
    \draw[thick,blue] (-2,0) circle (1cm);
    \draw[thick,red,->] (-3,0) -- (-2,0);
    %\draw[dashed] (-1,2) -- (-1,-2);
    %\draw[thick,blue] (-1.2,0) circle (0.2cm);
    %\draw[thick,red,->] (-1,0) -- (-1.2,0);
    \draw[dashed] (3,3) -- (3,-3);
    \draw[thick,blue] (2,0) circle (1cm);
    %\draw[dashed] (3,2) -- (3,-2);
    %\draw[thick,blue] (2.8,0) circle (0.2cm);
    \draw[thick,red,->] (3,0) -- (2,0);
    \draw[dashed] (6,3) -- (6,-3);
    \draw[thick,blue] (6,0) circle (1cm);
    %\draw[dashed] (1,2) -- (1,-2);
    %\draw[thick,blue] (1.2,0) circle (0.2cm);
    %\draw[thick,red,->] (1,0) -- (1.2,0);
    %\draw[blue] plot (\x,{\x});   
\end{tikzpicture}
\begin{tikzpicture}[scale=0.5]
    \begin{axis}[
        domain=-7:7,
        xmin=-7, xmax=7,
        ymin=-1.3, ymax=1.3,
        samples=400,
        axis y line=center,
        axis x line=middle,
    ]%\addplot[domain=0:2, thin, samples=100]{7*x^4}
     \addplot+[domain=-6:-3,blue,mark=none,samples=100] {-0.5*tanh(2*(-x-4.5))+0.5};
        \addplot+[domain=-3:3,blue,mark=none] {sin(deg(0.3333*0.5*3.1416*x+3.1416))}; %node[above] {$f(x)$};
        \addplot+[domain=3:6,blue,mark=none,samples=100] {-0.5*tanh(2*(-x+4.5))-0.5} node[above] {$f(x)$};
        \addplot+[domain=-7:7,red,dashed,mark=none] {1};
        \addplot+[domain=-7:7,red,dashed,mark=none] {-1};
        \addplot +[mark=none,red,dashed] coordinates {(-3, -1.3) (-3, 1.3)};
        \addplot +[mark=none,red,dashed] coordinates {(3, -1.3) (3, 1.3)};
    \end{axis}
     %\draw[dashed] (-3,1) -- (-3,-1);
\end{tikzpicture}
 \caption{\label{Randers-KropinaFig}Randers-Kropina metric $F$ in the Euclidean space $\R^2$ with $W_{(x,y)}=f(x)\partial_x$ and $|f|\leq 1$ as in the graph. The regions $x\leq -3$ and $x\geq 3$ are disconnected (non-connectable by wind curves), but the metric is complete since the forward and backward balls of radius $r$ are contained in the corresponding Euclidean balls with  radius $2r$ because $F(v)\leq 2|v|$ for every $v\in\R^2$. Moreover, the disconnected regions $x\leq -6$ and $x\geq 6$ are endowed with the Euclidean metric.}
 \end{figure}
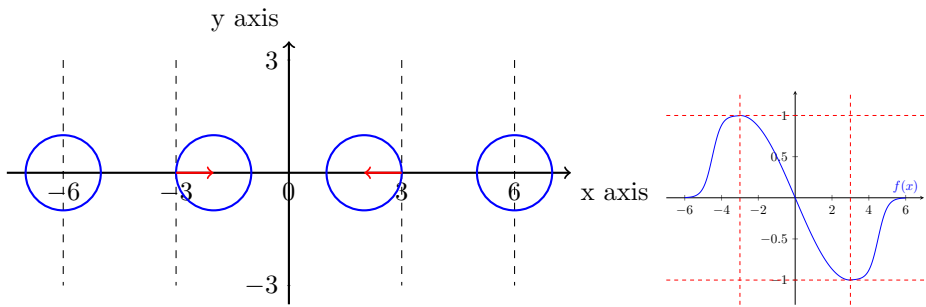
 These possibilities are very appealing from the SSTK viewpoint, as they are related to the existence of Killing horizons in globally hyperbolic spacetimes and other notions in the relativistic fauna. 
%\footnote{El que comente' por tele'fono, que tambien podri'a resultar u'til para otros propos'sitos. Caption: Randers Kropina on $M=\R^2$ ($g_R$ usual scalar product) 
%with $W_{(x,y)}=f(x)\partial_x$, $f(x) \equiv 0$ for $|x|\geq 3$, $f$ and increases an decreases as depicted?
%[increases from $0$ to $1$ on $[-3,-2]$,
% decreases from $1$ to $-1$ on $[-2,-1]$,  increases from $-1$ to $1$ on $[-1,1]$,  decreases from $1$ to $-1$ on $[1,2]$ and increases from $-1$ to $0$ on $[2,3]$.] It is geodesically complete and, thus, w-complete even though  wind curves starting at the strip  $x\in [-2,2]$ (resp. $(-2,2)$, $(-\infty, -1]$, $[1,\infty)$) cannot leave this region. Exceptional geodesics appear in $x=\pm 2, \pm 1$.}. 
 
%\section{Completeness and Hopf-Rinow Theorem for Randers-Kropina metrics}

\section{Wind Riemannian Structures of constant flag curvature}\label{s3}

\subsection{Randers and Kropina solutions} \label{s3.1}
 Randers manifolds of constant flag curvature have been completely classified by   Bao, Robles and Shen  \cite[Th. 3.1]{BRS}. This result has a local nature, even though it can be used to obtain models of inextendible Randers spaceforms (some of them  necessarily incomplete  as Finslerian manifolds). Their result can be  summarized as follows.
\begin{thm}\label{tbrs}
A Randers metric $F$ has constant flag curvature $\kappa$ if and only if its Zermelo data $(g_R,W)$  satisfy  the following:
\begin{enumerate}[(i)]
\item the metric $g_R$ has constant curvature $\kappa+\frac{1}{4}\mu^2$ for some constant $\mu$,
\item the wind $W$ satisfies $g_R(W,W)<1$ and it is $\mu$-homothetic for $g_R$, namely, ${\mathcal L}_W g_R=2\mu g_R$, where $\mathcal L$ is the Lie derivative.
\end{enumerate}
Moreover, the only complete simply connected Randers spaceforms are:
\begin{enumerate}
\item[(0)] if $\kappa=0$, those with Zermelo data given by the Euclidean metric and a parallel vector field with norm less than $1$,
\item[(-)] if $\kappa<0$, the hyperbolic space of constant curvature $\kappa$,
\item[(+)] if $\kappa>0$, those with Zermelo data given by the round sphere of radius $1/\sqrt{\kappa}$ and a Killing vector field with norm less than $1$.
\end{enumerate}
\end{thm}
Observe that given a vector field in the sphere, there is a multiple of it with norm less than $1$. However, completeness excludes the possibility of being properly homothetic for $W$, as well as most of the cases of Randers metrics that are not Riemannian for $\kappa\leq 0$.
About the  Kropina case with Zermelo data $(g_R,W)$,  Yoshikawa and Okubo  \cite[Th. 4]{YoOk07} solution becomes a natural extension of Bao et al.'s: $g_R$ must have constant curvature but now $W$ must be a homothetic  vector field satisfying $| W|_R\equiv 1$ (in particular, $W$  is then necessarily Killing,  see below). The corresponding explicit cases were obtained by Yoshikawa, and Sabau  \cite{YoSa}. 
For the convenience of the reader, we include some results on homothetic and Killing vector fields in the Appendix so that one can  rewrite the concrete Kropina solutions \cite{YoOk07, YoSa} as follows.
%\footnote{Comparese con \cite{YoSa}...}. 

\begin{thm}\label{th_KropClassif1}  A Kropina manifold $(M,F)$ of constant flag curvature is determined by Zermelo data $(g_R,W)$ which lie in one of the following two cases:

\begin{enumerate}[(i)]
\item $g_R$ is flat and $W$ is parallel and unit,
\item $g_R$ is locally isometric to an odd round sphere $S^{2m+1}(r)$ and $W$ is a unit Killing vector field, i.e. a unit Hopf vector field.
%; in particular,  $W$ is also a geodesic vector field ).
\end{enumerate}
\end{thm}
\begin{proof}
   By Proposition \ref{p4.4}, $W$ must be a Killing field. Moreover, Proposition \ref{p4.2} determines the two possibilities of the theorem. 
\end{proof}
Again the result  has a local nature,
%\br  \sout{However, in this case the notion of completeness makes sense by using geodesics (with velocities allowed by $F$). Moreover, the  completeness of such geodesics can be characterized in terms of the $F$-separation $d_F$ associated to any Kropina (or, in general, Randers-Kropina) metric. Indeed, $d_F$ is a natural generalization of the (possible non-symmetric) distance associated to any Finsler manifold, and the forward (resp. backward) completeness of the Kropina geodesics becomes equivalent to the compactenss of the closed forward (resp. backward) $d_F$-balls } \er \cite[Th. 4.9, Cor. 4.10]{CJSwind}. \br \sout{  Thus,   computing directly either the geodesics or the $d_F$-balls one obtains easily: the Kropina manifolds of constant flag curvature in Theorem \ref{th_KropClassif} are complete if and only if the metric $g_R$ is complete (see Theor. \ref{th_compl} below for a more general result), that is:}\er 
 but  all the local examples of Theorem \ref{th_KropClassif1} can be extended to complete simply connected models.

\begin{thm}\label{th_KropClassif2} The complete simply connected Kropina manifolds $(M,F)$ of constant flag curvature lie in one of the following two cases, up to isometries:

\begin{enumerate}[(i)]
\item $(M,g_R)=\R^n$ is flat and $W$ is a parallel, unit vector field of $\R^n$.
\item $(M,g_R)=\R^n= S^{2m+1}(r)$ and $W$ is a unit Hopf vector field.
\end{enumerate}
In particular, $W$ is also a complete geodesic vector field.
\end{thm}

\begin{rem}
As emphasized in subsection \ref{sec:balls}, completeness  always means ``forward and backward completeness''. Otherwise, further possibilities appear; for example, in the case $(i)$  the half space $x_n>0$ of $\R^n$ with $W=\partial_{x_n}$ is forward complete.
%\br The {\em forward} complete simply connected Kropina spaceforms \er\footnote{No se' si es conveniente usar esta palabra y, en caso afirmativo, si hablar de forward spaceform. Podemos decidirnos tambie'n (o no) a estudiar cua'les son las forward Randers spaceforms (la demostracio'n formal del remark parece fa'cil, habri'a que hacerlo y extenderlo) \bb Dejo esta nota a pie para nosotros, pero hay que borrarla \eb } admit another possiblity: $(M,g_R)$ is  the open half plane $x_n>0$ of $\R^n$ and $W=\partial_{x_n}$ (which is forward complete too). \footnote{\br La prueba de este resultado en $\R^n$ es simple porque sabemos calcular el futuro de un punto, pero y en la esfera con un campo de Hopf? \er}
\end{rem}

\subsection{Local classification for the WRS case}\label{s3.2}
 The following result relates the preservation of WRS's, Finslerian metrics and Zermelo data in the spirit of  \cite[Proposition 1]{BRS} 
(other background results about the indicatrices of wind Riemanian and wind Finslerian structures can be seen in Prop. 2.12, 2.13 and 2.46 of \cite{CJSwind}.

\begin{lemma}\label{l3.5}  Let $(M_i,\Sigma_i), i=1,2$ be two manifolds endowed with a WRS, each one with the associated Riemannian metric $g_i$,  wind $W_i$ and conic  pseudo-Finsler metrics  $F^i$,  $F^i_l$.  For any diffeomorphism $\phi: M_1\rightarrow M_2$, the following conditions are equivalent: 

 (a)  $\phi_*(\Sigma_1)=\Sigma_2$.

 (b)  $\phi$ is an isometry from $(M, g_1)$ to $(M,g_2)$ and $\phi_*(W_1)=W_2$.

 (c)  $\phi^*(F^2)=F^1$  and $\phi^*(F^2_l)=F^1_l$. 

 \smallskip
 
 In any of the previous cases, we will say that  $\phi$ is a {\em (WRS) isometry}, and  $(M_1,\Sigma_1)$ and $(M_2,\Sigma_2)$ are called {\em isometric}. 
\end{lemma}
\begin{proof}
 For $(a)\Rightarrow  (b) $ observe that $W_i$ is the centroid of the interior of $\Sigma_i$, because translating $\Sigma_i$ with $-W_{i}$ produces the  sphere $S_i$ of $g_i$, which is centered in the origin. Moreover, a linear map preserves the centroid, therefore $\phi_*(W_1)=W_2$, and then $\phi_*(S_1)=S_2$, which means that $\phi$ is  an isometry from $(M, g_1)$ to $(M,g_2)$. The implication   $(b)\Rightarrow (c)$   follows directly from the expression of $F^i$ and $F^i_l$ in terms of $g_i$ and $W_i$ (see %\cite[Prop. 2.58]{CJSwind}
 formulas \eqref{conicF}, \eqref{hmetric}, 
\eqref{lorentzF}) and  $(c)\Rightarrow (a)$  is straightforward. 
%As in \cite[Lemma 1.2]{BaChSh00}. For (2), notice that each conic Finler metric $F_i$ determines the Lorentz-Finsler one $(F_i)_l$ by \cite[Remark??(the number may have changed) 2.50]{CJSwind} 
\end{proof}
 The strong convexity (resp. concavity)   of $F$ (resp. $F_l$) allows one  to define the flag curvature for any flagpole in $A$ (resp. $A_l$). So, one has the natural notion. 
\begin{defi}
 A WRS has constant flag curvature if its associated conic metrics $F$ and $F_l$ has constant flag curvatures for all the  flagpoles
in $A$ and $A_l$, resp.
 \end{defi}
\begin{rem}\label{rjv}  The relation  between the  flag curvature of  (conic) pseudo-Finsler metrics in the same manifold and with indicatrices which are the same up to translation by an arbitrary homothetic vector of one of the pseudo-Finsler metrics  has been studied systematically in \cite{JV}. In particular, it is checked that if $(M,g_R)$ has constant curvature  and $W$ is a homothetic vector, then the corresponding WRS, $\Sigma=S_R+W$, has constant flag curvature, as in the Randers case \cite[Corollary 5.1]{JV}.
\end{rem}
\begin{thm}[Local classification of WRS]\label{tlocal} 
%Constant curvature plus homothetic, and cite the list in Th. \ref{tbrs} with no restriction on the norm of $W$.
A WRS $(M,\Sigma)$ has constant flag curvature $\kappa$ if and only if its Zermelo data $(g_R,W)$ satisfy  the following:
\begin{enumerate}[(i)]
\item the Riemannian metric $g_R$ has constant curvature $\kappa+\frac{1}{4}\mu^2$ for some constant $\mu$,
\item the wind $W$  is $\mu$-homothetic for $g_R$, namely, ${\mathcal L}_W g_R=2\mu g_R$, where $\mathcal L$ is the Lie derivative.
\end{enumerate}
\end{thm}
\begin{proof}
 First, observe that the local computation %\footnote{\bb NOTE FOR PROFESSOR BAO: We would need your help exactly for this point, i.e., to check that your computations hold even if $g_R(W,W)>1$, whenever they are restricted to the allowed domain $A$. Even more, taking into account the expression \eqref{conicF} for $F$ (which is valid no matter if the wind is strong, critical or mild) we think that your computations would also hold for $F$ on all $M$; in particular, this would re-prove the result by  Yoshikawa and Okubo \cite{YoOk07} (this is suggested in Remark \ref{ryo} but this is not necessary for any of our results here).\\ Nota para nosotros: fijate que prefiero olvidarme de la metrica $F_l$ para dejar ma's claro que necesitamos algo muy paracido a lo de Bao et al. \eb} 
 in \cite[Th. 3.1]{BRS} remains valid for the conic  Finsler metric  $F$ 
 in the open subset where $g_R(W,W)\not=1$, since in this subset $F$  
 is of Randers type. 
 Using this observation and Theorem ~\ref{th_KropClassif1} we deduce that locally $g_R$ has constant curvature and $W$ is $\mu$-homothetic in an open dense subset: the manifold $M$ except the points in the boundary  $\partial U$  of the region  $U$ where  $g_R(W,W)\not=1$. As both conditions are closed we deduce that $g_R$ is of constant curvature and $W$ is $\mu$-homothetic in every connected component  of $M\setminus \partial U$ and, by continuity, on all $M$. Thus, the result follows taking into account Remark \ref{rjv}. 
%Explicar que se puede seguir \cite{BRS} hasta su Th. 3.1, y Th 5.1.,
%trabajando con  la metrica conica que vale para todas las regiones. 
%
%Pese a que os calculos con la metrica conica debieran bastar, argumentamos que un dominio Kropina no se puede pegar con uno que no lo es porque la condicion $\parallel W\parallel =1$ no se puede
% extender analiticamente fuera de esa region (y si dos campos de Killing en una variedad diferenciable coinciden en un punto y sus primeras derivadas tambien, entonces coinciden en toda la variedad). 
\end{proof}
\begin{rem}\label{ryo}
We have used the Kropina classification \cite{YoOk07, YoSa} in our proof. However, one could consider  the expression  \eqref{conicF}  for the conic Finsler metric, which remains valid in all the regions, and try to carry out all the Randers-type computations in order to reprove the  Randers-Kropina case.
%Even though we have used \cite{YoOk07, YoSa} in the previous proof, one could consider  the expression  \eqref{conicF}  for the conic Finsler metric, which remains valid in all the regions, and carry out all the Randers-type computations in order to re-prove the \br Randers-\er Kropina case.  
\end{rem}
\subsection{The global result}\label{s3.3}

%In order to check that the inextendible simply connected wind Finslerian structures obtained naturally from the local result are also complete (according to our defintion) we can use our 
%Proposition 6.4(i): in the compact case for $M$ is obvious; otherwise, the completeness of the vector field $W$ should be enough to check the precompactness of the forward and bacward balls. But the completeness of $W$ should be straigthforward: on the one hand, every Killing vector field on a complete (semi-) Riemannian manifold is complete; on the other, the only homothetic vector field that we should take into account is the natural one in flat space, where everything must be straightforward.
In order to go from the local to the global result,  the following theorem on completeness becomes crucial, and the Example \ref{ex1} below is important  to know exactly what is going on.
\begin{thm}\label{th_compl}
Let $\Sigma$ be a WRS determined by Zermelo data $(g_R, W)$.
\begin{enumerate}[(i)]
\item If $\Sigma$ is complete, then $W$  is complete.

\item  If $g_R$ is complete and $W$ is a homothetic vector field, then $\Sigma$ is complete.
Moreover, if $W$ is properly homothetic (i.e., non-Killing), then $g_R$ is flat.
\end{enumerate}
\end{thm}

\begin{proof}

$(i)$ Notice that each integral curve $\rho$  of $W$ is a wind curve. Thus, if $\rho$ were forward or backward incomplete, then a closed forward or backward ball (centered at $\rho(0)$ and with radius equal to the finite length of $\rho$ towards $+\infty$ or $-\infty$ would be non compact  (recall that, by Theorem \ref{c63},  the geodesic completeness  of a WRS   is equivalent to the precompactness of its  balls).

$(ii)$ Assuming that $W$ is complete, the first assertion   follows from  Theorem~\ref{JVresult}.  However, the completeness of $W$ can be deduced from the completeness of $g_R$. Indeed, when $W$ is Killing, such a result follows from the fact that the flow of $W$ must preserve a small closed ball along all its integral curves (for a more general result valid even  in the semi-Riemannian case, see  \cite[Prop. 9.30]{Oneill83}). In the case that $W$ is properly homothetic,  Tashiro \cite[Th. 4.1]{TASHIRO} (which summarizes a conclusion from previous work by Kobayashi \cite{Koba}, Nomizu \cite{Nomizu} and Yano and Tagano \cite{YT})  ensures
%\footnote{\bb PARA NOSOTROS The proof is not his. Namely, consider  the universal cover $(\tilde M,\tilde g_R)$  of $(M,g_R)$ and the lifted homothetic vector field $\tilde W$. By taking  de Rham  decomposition, $(\tilde M,\tilde g_R)$ can be written as a product $(N\times \R^k, h\oplus g_0)$ with $(N,h)$ a product of irreducible non-flat manifolds (and thus, not admitting a homothetic non-Killing vector field). Even more, by NOMIZU the projection in each irreducible part is affine etc...\eb}   
that $(M,g_R)$ is flat (thus, obtaining the second assertion). Therefore, $(M,g_R)$ is globally covered by $\R^n$. But the homothetic vector fields  in $\R^n$ are well known and they must be complete\footnote{ Notice that they are affine, and all the affine vector fields of $\R^n$ has affine natural coordinates. Thus, their completeness follows by direct integration (or by applying general results such as \cite[Theorem 1]{SaFIELDS}.)  }.  
\end{proof}
It is worth emphasizing that the completeness of $\Sigma$ does {\em not} imply the completeness of $g_R$ {\em even in the Randers case}.
\begin{exe}\label{ex1}
Consider on $\R^+$ the Randers metric $R$ determined by Zermelo's $(g_R=dx^2, W=f\partial_x)$, where $f:\R^+\rightarrow \R$ is any  function satisfying the following.  Take the intervals $I_k=[2^{-(k+1)}, 2^{-k}], k=0, 1,\dots$  and put
$$
f(x)=
\left\{\begin{array}{lrl}
2^{-(4k+1)}-1 & & \hbox{on} \; I_{4k} \\
1-2^{-(4k+3)} & & \hbox{on} \; I_{ 4 k+2} 
\end{array}\right.
\qquad 
\hbox{so that} 
\left\{\begin{array}{lrl}
R(\partial_x)=2^{4k+1} & & \hbox{on} \; I_{4k} \\
R(-\partial_x)=2^{4k+3} & & \hbox{on} \; I_{4 k+2} 
\end{array}\right.
$$
with no  restriction on $f$ outside the above intervals except being $C^\infty$ and $|f|<1$ on all $\R^+$. As the length of each interval $I_k$ is $2^{-(k+1)}$,  the $R$-length $L$ of the  curves $\gamma_\pm(s)=\pm s$ satisfies  $L(\gamma_+|_{I_{4k}})=1$, $L(\gamma_-|_{I_{4k+2}})=1$. This implies the (forward and backward) completeness of  $R$, in spite of the incompleteness of 
$g_R$.\footnote{An alternative way of producing this  type of examples  appears considering the spacetime viewpoint. Take any SSTK splitting $\R\times M$ with a  Cauchy hypersurface $\{0\}\times M$ which inherits an incomplete Riemannian metric $g_R$ (a Cauchy hypersurface $S$ satisfying this property can be easily constructed in Lorentz-Minkowski, and moving $S$ with the flow of the natural timelike parallel vector field $K=\partial_t$ one obtains the required SSTK splitting). The completeness for the induced Randers metric  (obtained from the incomplete  $g_R$ and the wind $W$ in \eqref{eg} and \eqref{egbis}) is a consequence of the fact that $S$ is Cauchy (see \cite[Theorem 4.4]{CJS} or the more general  \cite[Theorem  5.11, part (iv)]{CJSwind}).}  
\end{exe}
 Now, we are ready to prove our main result. 
\begin{thm}\label{thm:globalresult}
The complete simply connected  WRS's   with constant flag curvature  lie in one of the following two exclusive cases, determined by Zermelo data: 
%\begin{itemize}
%\item 
\begin{enumerate}[(i)]
\item  $(M,g_R)$ is a  model space of constant curvature  and $W$ is any of its Killing vector fields.
%\item 

\item  $(M,g_R)$ is isometric to 
$\R^n$  and $W$ is  a properly homothetic (non-Killing) 
vector field.
\end{enumerate}
\end{thm}
\begin{proof} The fact that the two WRS's  above have constant flag curvature comes from the local result (Theorem \ref{tlocal}) and their completeness from part $(ii)$ of  Theorem~\ref{th_compl}.

Conversely,  if $\Sigma$ is a complete simply connected WRS with constant flag curvature,   from the local result  we  know that locally $g_R$ has constant curvature (and, so, this happens globally) and $W$ is %(locally and then globally) 
a homothetic vector field. Now, we have to check the completeness of $g_R$. Recall first that $W$ must be complete (Theorem \ref{th_compl}-$(i)$) and consider two cases. 

In the case that $W$ is Killing for $g_R$, then it is also Killing for $\Sigma$, in the sense that its flow is composed by isometries (indeed, it satisfies the condition (b) in Lemma~\ref{l3.5}). Notice that $g_R$ can be seen as the WRS obtained by the translation   $\Sigma-W$, and  the completeness of this WRS follows   from Theorem~\ref{JVresult},  since $\Sigma$ is complete  and,  as commented above, $W$ is Killing for $\Sigma$ and complete because it is Killing for $g_R$. 

In the case that $W$ is properly homothetic,  $g_R$ must be flat (see part $(i)$ of Proposition~\ref{p4.4}). Let us first see that the completeness of $W$ and $\Sigma$ implies that  $W$ must have a zero. Choose any point $p_0$ and the integral curve $\gamma$ of $W$ through $p_0$. We can assume that the $W$-flow $\phi_t$  is homothetic of ratio $e^{-\mu} <1$ for $t=-1$ (otherwise, take the flow of $-W$  and recall the comment about the reverse metric before Theorem \ref{c63}).  %; recall that $\Sigma=S_R+W$ is complete if and only if so is its reverse $\Sigma=S_R-W$). 
Then,  $\{p_k:=\phi_{-k}(p_0)\}_{k\in \N}$ is a Cauchy sequence for $g_R$ and the $g_R$-length of $\gamma|_{(-\infty,0]}$ is finite. Moreover, for some 
%$k_0\in\N$,  
$t_0\leq 0$ one has that $|W|_R<1/2$ at $q=\phi_{t_0}(p_0)$ and, then, also on all $\phi_t(p)$ for $t\leq t_0$. Thus, $|W|_R<1/2$ on a neighborhood $U$ of $\gamma ((-\infty,t_0])$ and
$\Sigma$ becomes a Randers metric $F$ on $U$. From the definition of $F$, one has 
$$
F(v_x) \leq \frac{| v_x|_R}{1-| W_x|_R} \leq  2 | v_x|_R $$ 
on $U$.  Thus, the $F$-length of $\gamma|_{(-\infty,0]}$ is also finite and $\{p_k\}_k$ 
becomes a Cauchy sequence also  for $F$. 
 Observe that there exists $t_0<0$ such that $\gamma|_{(-\infty,t_0)}$
 is contained
in the backward ball $B^-_{\Sigma}(\gamma(0),\ell_F(\gamma))$ (recall that $F$ is Randers in the neighborhood $U$ of $\gamma|_{[-\infty,t_0]})$), and this ball is, by the hypothesis of completeness
of $\Sigma$ and Theorem \ref{c63}, precompact. 
%\br Observe that as $\gamma$ is not a $\Sigma$-geodesic (otherwise it would be complete by hypothesis) there exists $t_1<0$ such that $\gamma(t_1)\in B^-_{\Sigma}(\gamma(0),r)$. In particular, there exists a curve $\beta:[t_1,t_2]\rightarrow M$ from $\gamma(t_1)$ to $\gamma(0)$ such that $\ell_F(\beta)<r<\ell_{F_l}(\beta)$. Now the concatenation of $\gamma|_{(-\infty,t_1]}$ and $\beta$ gives a curve $\tilde{\gamma}$ such that $\ell_F(\tilde\gamma)<\ell_{F_l}(\tilde\gamma)$. Choose $\tilde{r}$ such that $\ell_F(\tilde\gamma)<\tilde{r}<\ell_{F_l}(\tilde\gamma)$.  Then it follows by continuity that there exists $t_3$ such that $\ell_F(\tilde\gamma|_{[t_3,t_2]})<\ell_F(\tilde{\gamma})<\ell_{F_l}(\tilde\gamma_{[t_3,t_2]})$ and if $t\leq t_3$,  $\tilde{\gamma}(t)=\gamma(t)$ and it is contained in the backward c-ball $\hat{B}^\pm_{\Sigma}(\gamma(0),\ell_F(\tilde\gamma))$, which, by the hypothesis of \er
 Then the Cauchy sequence $\{p_k\}_k$ 
%the curve has and endpoint 
 will have a limit $p_\infty$ and 
%$W_{p_\infty} . 
%(and the 
%Cauchy sequence $\{p_k\}_k$ converges to it). 
$W_{p_\infty}=0$ (on $p_k$, $|W|_R$ is smaller than $e^{-k\mu}\rightarrow 0$), as required. 
Now, for some small $\epsilon>0$, the closed $g_R$-ball $\bar B(p_\infty,\epsilon)$ 
%of center $p_\infty$ and radius $\epsilon$, $B_\epsilon$ is compact. 
is compact and, thus, so is $\phi_k(\bar B(p_\infty,\epsilon))=\bar B(p_\infty,e^{\mu k}\epsilon)$ for all $k>0$, which implies  that the flat manifold has to be isometric to $\R^n$. 
%choose $p_\infty$ as the origin $0$ of orthonormal 
%coordinates on   $M$ defined on an open  ball centered at $0$ with  maximum  radius $r_+$; our aim is to prove $r_+=\infty$. Let $r\partial_r/2$ be the standard homothetic vector in the 
%associated spherical coordinates. 
%If $W$ is $
%\sigma$-homothetic
%then $X:=W-\sigma r \partial_r /2$ is a Killing vector field that vanishes at the origin; so, at each point $x\neq 0$, $X$ is tangent to the Euclidean   sphere $S_r$ centered at $0$ of radius $r=r(x)$. Notice that $S_r$ is a topological and metric sphere for $0<r<r_+$, and any integral curve $\gamma$  of $W$ with limit point 0 crosses transversally all the spheres 
%$S_r, r$, thus providing an homeomorfism between any two of such spheres $S_{r_i}, i=1,2$, with $0<r_1<r_2<r_+$.  Indeed, each $\gamma$ satisfies $g_R(\partial_r,\gamma')=r/2$ and, 
%thus, each $\gamma$ not only crosses all the spheres but also must have a limit point in a 
%subset $S\subset M$ in the closure of the neighborhood, providing also a continuous 
%bijective map $\psi: S_{r_+/2}\rightarrow S$ which maps each $x \in S_{r_+/2}$ to the unique point $y\in S$ crossed by the same integral curve of $W$.  

Finally,  the classical models of simply connected spaceforms plus part $(ii)$ of Theorem~\ref{th_compl} imply that the two stated cases are the unique possibilities of global representation for a WRS spaceform.

\end{proof}

\begin{rem} In the proof of completeness  for homothetic $W$ above, it was crucial  that  the completeness of $\Sigma$ implied the existence of a point invariant by  the $W$-flow. Indeed, $\R^n$ minus the origin (or minus any closed set of radial half lines) is an example of flat incomplete manifold with a complete homothetic vector field, namely, $r\partial_r$.

Thus, in the case $(ii)$ one can choose the origin $0$ of $\R^n$ as the unique point where $W$ must vanish. Then, for some  $ \mu  \neq 0$, the vector field
$Y:=W-\mu r \partial_r$ is Killing, vanishing  at 0, and tangent to the spheres centred at 0.  
\end{rem}

\section{Appendix:  results on homothetic and Killing fields}\label{s4}

 Next,  some results on Killing and homothetic vector fields  spread in the literature, which become relevant for our discussions, are collected for the convenience of the reader (see also the summary in \cite{RomSan02} for other related results).

 \subsection{Global and local results} \label{s4.1} Global results 
(to be taken into account in the global classification of constant flag WRS) are the following. 
\begin{prop}
 Let $(M,g_R)$ be a  complete Riemannian manifold. 

\begin{enumerate}[(i)]
\item If $W$ is a homothetic vector field and $g$ is not flat, then $W$ is Killing.

\item  If $W$ is a homothetic (or, with more generality, an affine) vector field with bounded (pointwise) norm,  then $W$ is  Killing.
Moreover, in the case that $g_R$ is flat then $W$ is parallel. %generates a 
%one-parameter group of translations.

\item If $g_R$ has constant curvature $\kappa\leq 0$ and $W$ is  homothetic with bounded norm then $\kappa=0$ and $W$ is parallel.  
\end{enumerate}
\end{prop}
\begin{proof}
$(i)$ See \cite[Lemma VI.2 (p. 242)]{KoNo1}.

$(ii)$ The first assertion is a well known theorem by Hano \cite[Theorem 2]{Han} (see also \cite[Theorem VI.3.8]{KoNo1}), for the second one see \cite[Lemma 4]{Han}.

$(iii)$ We know from  $(ii)$  that $W$ is Killing and, in the case $\kappa=0$, it is parallel. So, passing to the universal covering, it is enough to prove that the norm of any Killing vector field of the  hyperbolic space $\HH^n$ is unbounded. Looking $\HH^n$ as 
the upper component of the unit timelike vectors of Lorentz-Minkowski $\LL^{n+1}$, its Killing vector fields can be seen as the subalgebra of the Killing fields of $\LL^{n+1}$ which vanish at the origin (the flow of such Killing fields  leaves  invariant the unit vectors, so, they 
preserve and are tangent to $\HH^n$).  If $S^*$ is one of such vector fields, it can be  constructed by taking each skew-adjoint linear matrix $S$ of $\LL^{n+1}$ 
and putting $S^*_p=(Sp)_p$ for all $p\in \LL^{n+1}$, where $(Sp)_p$ is obtained by regarding  $Sp\in 
\LL^{n+1}$ as a vector in $T_p\LL^n$ (see for example 
\cite[Example 9.29]{Oneill83}). Notice that these skew-adjoint matrices are:
$$
S= \left(
\begin{array}{c | c}
0 & b_1  \dots b_n \\
\hline
\begin{array}{c}
b_1 \\ \vdots \\ b_n 
\end{array}
& A
\end{array} \right)
$$
where $A$ is any skew-symmetric $n\times n$ matrix ($a_{ij}=-a_{ji}$). Then, choosing the points $p_i^\pm(\lambda)= (\sqrt{1+\lambda^2},0,\dots, 0,^{(i)} \pm \lambda, 0, \dots , 0)^t\in \HH^n$ for $i=1, \dots, n$, and making $\lambda\rightarrow +\infty$ one checks that the unique possibility to bound the induced Riemannian norm on all $Sp$ is to put $S=0$. 
\end{proof}

\begin{rem}\label{r_cocotazo}

The previous result shows that, under the completeness of $g_R$, one obtains necessarily a WRS which is not Randers in the following cases: 

 (a) when the wind vector $W$ is properly homothetic,

(b) in the case of constant curvature $k\leq 0$, when $W$ is Killing but non-parallel.

 Indeed, the Randers models of constant flag curvature obtained in \cite{BR, BRS} fail to be complete in these cases. For the case of positive curvature, of course, one can take any Killing vector field as the wind $W$ and multiply it by a constant so that it becomes mild  everywhere (or if it has no zeroes, strong everywhere).
\end{rem}
 The following local result (with no assumption of completeness) bounds the possible cases of WRS with constant flag curvature.
\begin{prop}\label{p4.4} Let $Z$ be a homothetic vector field  in a  Riemannian  manifold of constant curvature $\kappa$:
\begin{enumerate}[(i)]
\item When $ \kappa \neq 0$,  $Z$  is Killing. 

\item  When $ \kappa  =0$,   $Z$ 
 has  constant norm if and only if  $Z$ 
is  parallel. 
\end{enumerate}
\end{prop}

\begin{proof}
$(i)$  As observed in \cite[\S 3.2.1]{BRS}, the flow $\psi_t$ of a $\mu$-homothetic field sends the constant curvature $k$ to $e^{2\mu t}k$, and then by the hypothesis of constant curvature, $e^{2\mu t}=1$ and $\mu=0$. 

$(ii)$   Observe that any homothetic vector field is affine. Then,  by a direct local computation  in natural orthonormal coordinates of $\R^n$, any affine vector field $Z$  and its norm are written as 
$$
Z=\sum_{j=1}^n \left(a_j +\sum_{i=1}^n a^i_j x_i\right) \partial_j ,\qquad | Z|^2= \sum_{j=1}^n \left(a_j +\sum_{i=1}^n a^i_j  x_i\right)^2 ,
$$
and the matrix $(a_i^j)$  vanishes by equating $\frac{\partial^2}{\partial x_k^2} | Z|^2=0$ for  $k=1,\dots , n$.
  \end{proof}
\begin{rem}
 Observe that the part $(i)$ of Proposition \ref{p4.4} can be generalized in several directions. For example,   Knebelman \cite[Theorem 4 and below]{Kne45} showed that any affine vector field in a non-Ricci flat Einstein manifold is Killing, and  Knebelman and Yano \cite[Theorem 2]{KneYan60}
 that,  for non-zero
constant  scalar curvature,  any homothetic vector field is Killing. 
\end{rem}
 \subsection{Further results for Killing vectors  with constant norm} \label{s4.2} Once the homothetic case was ruled out by Proposition \ref{p4.4}, the hypothesis of constant norm for  Killing vector fields becomes essential for the possible Kropina models. So, the following results are in order. 

\begin{lemma}\label{p4.1} A  Killing vector field $Z$ on a Riemannian manifold  has constant norm if and only if it is geodesic. In this case 
  $g(\nabla_X Z,\nabla_Y Z) = R(X,Z,Z,Y)$ for all vector fields $X,Y$ on $M$. 
\end{lemma}
\begin{proof}  The first assertion follows from $g(\nabla_ZZ,X)=-g(\nabla_XZ,Z)=- \frac 12 X(g(Z,Z))$.
The second one is a straightforward computation  (see 
\cite[Lemma 3, Prop. 1] {BerNik09}).
\end{proof}

\begin{prop}\label{p4.2} Let $Z$ be a Killing vector field of constant norm 1 on a Riemannian manifold:
\begin{enumerate}[(i)]
 \item $Z$ cannot exist for negative curvature, and must be parallel for 0 curvature.

\item $Z$ cannot exist in manifolds of positive curvature of even dimension.  %in round spheres  (or, with more generality, compact manifolds of positive curvature)  of even dimension. 

\item In round spheres of dimension odd $S^{2n+1}(r)$, the vector field $Z$ is a Hopf vector field, i.e., 
one of the natural unit Killing vector fields  tangent to the fiber in 
Hopf fibration. %$S^{2n+1}\rightarrow CP^{n}$.
\end{enumerate}

\end{prop}
 \begin{proof} $(i)$  Using  the formula in  the previous lemma 
 $$
 g(\nabla_X Z,\nabla_X Z) 
= R(X,Z,Z,X).  $$
 In the case of negative curvature, a contradiction is obtained by choosing $X\neq 0$ orthogonal to $Z$. The flat case follows applying this formula to all $X$. 

 $(ii)$
%Writing the previous formula | |= R( ), and recall that \nabla... restricted to the orthogonal of $Z$. As this orthogonal is odd dimensional, it has...   
  %Put $f=g(Z,Z)/2$ and  recall\: $$0= \hbox{Hess} f(V,V)= | \nabla_V  Z|^2 - R(V,Z,Z,V).$$
 As $ \nabla Z  $ is a skew adjoint operator  ($X\mapsto \nabla_XZ$)  that  can be restricted to the (odd dimensional) orthogonal to $Z$, it has a 0 eigenvalue,  and   one of  such  eigenvectors  $X$  yields a contradiction in the formula above.\footnote{This proof follows the spirit of Berger's one \cite{Berger} to prove that  any Killing vector field on an even dimensional compact manifold of positive curvature must have a zero.}

$(iii)$ The proof (for  radius  $r=1$)  is  in \cite[Theorem 1]{Wieg} (see also  \cite[\S 3]{GilLli}).
\end{proof}

%\begin{rem} 
Recall that the exception in the case $(iii)$  corresponds to  a natural unit Killing vector field  tangent to the fiber in Hopf fibration   which is extended to  a fibration in any odd dimension  $S^{2n+1}\rightarrow CP^{n}$  (see for instance \cite[Vol. II, p. 135]{KoNo2}) thus providing all the possible examples in the odd dimensional case.
%\end{rem}

\section*{Acknowledgements}

%The authors warmly acknowledge  ... for helpful conversations on the topics of this
%paper. \sout{ and the anonimous referee for his interesting comments.}
The authors warmly acknowledge  to Professor D. Bao (San Francisco State University) his prompt and kind comments on the paper. Especially, we are very grateful  he confirmed our assumption that the computations behind \cite{BR, BRS} with $|W|_R<1$ remain valid
for the strict inequality $|W|_R>1$,    
%to check the validity of 
%5on the computations for constant flag curvature when the %metric $h$ is {\em non-degenerate}, even if not positive %definite, 
a crucial step  in the proof of Theorem \ref{tlocal}.   Partially
supported by Spanish  MINECO/FEDER project reference
MTM2015-65430-P,  (MAJ) and MTM2013-47828-C2-1-P (MS).

\end{document}